\documentclass[a4paper,twoside,american,english,DIV12]{scrartcl}
\usepackage[T1]{fontenc}
\usepackage{verbatim}
\usepackage{mathtools}
\usepackage{amsthm}
\usepackage{amsmath}
\usepackage{amssymb}
\PassOptionsToPackage{normalem}{ulem}
\usepackage{ulem}

\makeatletter

\numberwithin{equation}{section}
\numberwithin{figure}{section}
\theoremstyle{plain}
\newtheorem{thm}{\protect\theoremname}
  \theoremstyle{definition}
  \newtheorem{defn}[thm]{\protect\definitionname}
  \theoremstyle{plain}
  \newtheorem{prop}[thm]{\protect\propositionname}
  \theoremstyle{remark}
  \newtheorem{rem}[thm]{\protect\remarkname}

\usepackage[english]{babel}     
\usepackage[T1]{fontenc}    

\usepackage{a4wide}        
\usepackage{scrpage2}
\usepackage{tu-preprint}
\usepackage{amsmath}
\usepackage{euscript}
\allowdisplaybreaks[1] 
\usepackage{amsfonts}

\newcommand{\hide}[1]{}

\newcommand*{\dive}{\operatorname{div}}

\newcommand*{\Grad}{\operatorname{Grad}}
\newcommand*{\Div}{\operatorname{Div}}

\newcommand*{\grad}{\operatorname{grad}}

\newcommand*{\ii}{\mathrm{i}}

\DeclareMathAccent{\Circ}{\mathalpha}{operators}{"17}
\newcommand{\interior}[1]{\Circ{#1}}
\renewcommand{\Im}{\operatorname{\mathfrak{Im}}}
\renewcommand{\Re}{\operatorname{\mathfrak{Re}}}

\newcommand{\oi}[2]{\left]#1,#2 \right[}
\newcommand{\rga}[1]{\left]#1,\infty  \right[}

\renewcommand{\tilde}{\widetilde}
\renewcommand*{\epsilon}{\varepsilon}
\renewcommand*{\rho}{\varrho}

\makeatother

\usepackage{babel}
  \addto\captionsamerican{\renewcommand{\definitionname}{Definition}}
  \addto\captionsamerican{\renewcommand{\propositionname}{Proposition}}
  \addto\captionsamerican{\renewcommand{\remarkname}{Remark}}
  \addto\captionsamerican{\renewcommand{\theoremname}{Theorem}}
  \addto\captionsenglish{\renewcommand{\definitionname}{Definition}}
  \addto\captionsenglish{\renewcommand{\propositionname}{Proposition}}
  \addto\captionsenglish{\renewcommand{\remarkname}{Remark}}
  \addto\captionsenglish{\renewcommand{\theoremname}{Theorem}}
  \providecommand{\definitionname}{Definition}
  \providecommand{\propositionname}{Proposition}
  \providecommand{\remarkname}{Remark}
\providecommand{\theoremname}{Theorem}

\begin{document}
\institut{Institut f\"ur Analysis}

\preprintnumber{MATH-AN-04-2015}

\preprinttitle{A Note on a Two-Temperature Model in Linear Thermoelasticity.}

\author{Santwana Mukhopadhyay, Rainer Picard, Sascha Trostorff, \& Marcus Waurick} 

\makepreprinttitlepage

\selectlanguage{american}%
\setcounter{section}{-1}

\date{}

\selectlanguage{english}%

\title{A Note on a Two-Temperature Model in Linear Thermoelasticity.}

\author{S. Mukhopadhyay%
\thanks{\emph{S. Mukhopadhyay:} Department of Mathematical Sciences, Indian
Institute of Technology (BHU), Varanasi -221005, India: Corresponding
author: E-mail: mukhosant.apm@iitbhu.ac.in %
}, R. Picard\textit{\emph{,}} S. Trostorff, M. Waurick\textit{\emph{}}%
\thanks{\emph{R. Picard}\textit{\emph{,}}\emph{ S. Trostorff, M. Waurick}\textit{\emph{:
Institute for Analysis, Faculty of Mathematics and Sciences, TU Dresden,
Germany.}}%
}}

\maketitle
\setcounter{section}{-1}
\abstract{\textbf{Abstract.} We discuss the so-called two-temperature model in linear thermoelasticity and provide a Hilbert space framework for proving well-posedness of the equations under consideration. With the abstract perspective of evolutionary equations, the two-temperature model turns out to be a coupled system of the elastic equations and an abstract ode. Following this line of reasoning, we propose another model being entirely an abstract ode. We highlight also an alternative way for a two-temperature model, which might be of independent interest.}

\textbf{MSC(2010):} {35F15, 35M32, 35Q74}

\textbf{Keywords:} {Evolutionary Equations, thermoelasticity, Two-Temperature Model, Coupled Systems}

\date{}

\section{Introduction}

Chen and Gurtin \cite{Chen-Gurtin-1968} and Chen et al. \cite{CGW-1968,CGW-1969}
have given the formulation of the theory of heat conduction related
to a deformable body which is based on two different temperatures.
Here the first one is the conductive temperature, $\phi$ and the
other one is the thermodynamic temperature, $\theta$. Chen et al.
\cite{CGW-1968} discussed that these two temperatures are equal in
absence of heat supply in the case of time independent situations
and the difference between these two temperatures is proportional
to the heat supply, where, as generally in the cases of time dependency,
these two temperatures are different. Before these studies, by doing
the study of transient coupled thermoelastic boundary value problem
in half space, Boley and Tolins \cite{key-11} gave the conclusion
that the strain and two temperatures are found to have explanation
in the form of a wave plus a response taking place immediately through
the body. The uniqueness and reciprocity theorems for the two- temperature
thermoelasticity theory in case of a homogeneous and isotropic solid
was reported by Lesan \cite{key-12}. Subsequently, investigations
were carried out on the basis of this theory by several researchers
like, Warren and Chen \cite{key-13}, Warren \cite{key-14}, Amos
\cite{key-15} etc. This theory (2TT) has drawn attention of researchers
in recent years and some specific features of this theory are reported
(see \cite{key-16,key-17,key-18,key-19,key-20,key-21,key-22,key-3},
\cite{key-3-1,quint-2004} and the references there-in). 

A structural formulation for linear material laws in classical mathematical
physics was introduced by Picard \cite{Picard2009} who considered
a class of evolutionary problems which covers a number of initial
boundary value problems of classical mathematical physics. The corresponding
solution theory is also established in \cite{Picard2009}. Prior to
this, Picard \cite{Picard2005} also reported the structural formulation
for linear thermo-elasticity in nonsmooth media. Recently, Mukhopadhyay
et al.~\cite{key-5} have studied various models of thermoelasticity
theory and have shown that the these models can be treated within
the common structural framework of evolutionary equations and considering
the flexibility of the structural perspective they obtained well-posedness
results for a large class of generalized models allowing for more
general material properties such as anisotropies, inhomogeneities,
etc. It should be noted that evolutionary equations in the form just
discussed have also been studied with regards to homogenization theory,
see e.g.~\cite{Waurick2012,Waurick2012a,Waurick2013}. The aim of
this note is to analyze the two-temperature thermoelastic model given
by Chen and Gurtin \cite{Chen-Gurtin-1968} as a first order system
within the framework of evolutionary equations, see e.g.~\cite{key-4}.
The models of thermoelasticity we shall discuss have been originally
conceived as constant coefficient models. There is little harm in
this assumption at this point, since we shall dispose of this simplification
completely, when we discuss more general models in the last section.
An alternative two-temperature thermoelastic model is proposed in
which we can avoid involving roots of an unbounded operator. It is
believed that the general perspective on two-temperature thermoelasticity
to be presented may shed some new light on the theory of homogenization
of such models.

In section \ref{sec:fap}, we discuss the functional analytic background
needed, for discussing the two-temperature model. Section \ref{sec:2tt}
discusses the two-temperature model in detail. In this section, we
will also give a suitable Hilbert space framework allowing for well-posedness
of the respective equation. An observation in section \ref{sec:2tt}
is that the heat equation part is replaced by an abstract ode with
infinite dimensional state space. More precisely, in the heat equation
part the \emph{only} unbounded operator involved is the time-derivative.
Having realized this property of the two-temperature model, we propose
in the two concluding sections \ref{sec:2tt-2s} and \ref{sec:a-2tt}
alternative systems of thermoelasticity. The first one being entirely
an abstract ode in the sense just discussed. The second one describes
a possible alternative model, which does not involve square roots
of operators.

\section{Functional Analytic Preliminaries\label{sec:fap}}

In this section, we shall elaborate on some standard concepts in functional
analysis needed in the following. Most frequently, we will have occasion
to use the square root and the modulus of an operator:
\begin{defn}
Let $H_{0}$, $H_{1}$ be Hilbert spaces. Let $C\colon D(C)\subseteq H_{0}\to H_{0}$
be a non-negative definite, selfadjoint operator, that is, for all
$\phi\in D(C)$ we have $\langle\phi,C\phi\rangle\geq0$ and $C=C^{*}$,
then $\sqrt{C}$ is defined as the unique non-negative definite operator
satisfying $\sqrt{C}\sqrt{C}=C$. For $A\colon D(A)\subseteq H_{0}\to H_{1}$,
a closed and densely defined linear operator, we define the modulus
of $A$, $\left|A\right|$, by
\[
\left|A\right|\coloneqq\sqrt{A^{*}A}.
\]

\end{defn}
Recall that $D(A)=D(\left|A\right|)$ and that $\|A\phi\|=\|\left|A\right|\phi\|$
for all $\phi\in D(A)$. We record the following standard fact:
\begin{prop}
\label{prop:comm_with_sqr}Let $H_{0}$, $H_{1}$ be Hilbert spaces,
$A\colon D(A)\subseteq H_{0}\to H_{1}$ densely defined, closed, linear.
Then
\[
\overline{\left(\sqrt{1+\left|A^{*}\right|^{2}}\right)^{-1}A}=A\left(\sqrt{1+\left|A\right|^{2}}\right)^{-1}\in L(H_{0},H_{1}).
\]
with $\left\Vert A\left(\sqrt{1+\left|A\right|^{2}}\right)^{-1}\right\Vert \leq1$.\end{prop}
\begin{proof}
Let $A=U\left|A\right|$ with a partial isometry $U$, being in particular
a contraction, we have by the spectral theorem, $\left(\sqrt{1+\left|A\right|^{2}}\right)^{-1}\phi\in D(\left|A\right|)$
for all $\phi\in H_{0}$, and thus,
\begin{align*}
\left\Vert A\left(\sqrt{1+\left|A\right|^{2}}\right)^{-1}\phi\right\Vert  & =\left\Vert \left|A\right|\left(\sqrt{1+\left|A\right|^{2}}\right)^{-1}\phi\right\Vert \\
 & \leq\left\Vert \phi\right\Vert ,\qquad(\phi\in H_{0}),
\end{align*}
establishing the boundedness and the norm-estimate of the operator
$A\left(\sqrt{1+\left|A\right|^{2}}\right)^{-1}$. As the operator
$\left(\sqrt{1+\left|A^{*}\right|^{2}}\right)^{-1}A$ is densely defined,
for the asserted equality in the proposition, it suffices to establish
the inclusion 
\begin{equation}
\left(\sqrt{1+\left|A^{*}\right|^{2}}\right)^{-1}A\subseteq A\left(\sqrt{1+\left|A\right|^{2}}\right)^{-1}.\label{eq:incl}
\end{equation}
Next, we prove \eqref{eq:incl}: For this, by induction, we show the
inclusion
\begin{equation}
(1+AA^{*})^{-n}A\subseteq A(1+A^{*}A)^{-n}\quad(n\in\mathbb{N}).\label{eq:powers_of_res}
\end{equation}
For proving the latter inclusion for $n=1$, observe that for $\phi\in D(AA^{*}A)$,
we have
\[
(1+AA^{*})A\phi=A(1+A^{*}A)\phi.
\]
Hence, substituting $\psi\coloneqq(1+A^{*}A)\phi$, we get
\[
A(1+A^{*}A)^{-1}\psi=(1+AA^{*})^{-1}A\psi.
\]
So, for every $n\in\mathbb{N}$ the inductive step can be shown as
follows:
\begin{align*}
(1+AA^{*})^{-(n+1)}A & =(1+AA^{*})^{-n}(1+AA^{*})^{-1}A\\
 & \subseteq(1+AA^{*})^{-n}A(1+A^{*}A)^{-1}\\
 & \subseteq A(1+A^{*}A)^{-n}(1+A^{*}A)^{-1}\\
 & =A(1+A^{*}A)^{-(n+1)}.
\end{align*}
For the proof of \eqref{eq:incl}, we recall that for every real number
$x>0,$ with $|x|<1$ the binomial series gives
\begin{equation}
\sqrt{1+x}=\sum_{n=0}^{\infty}\binom{1/2}{n}x^{n}.\label{eq:bin}
\end{equation}
Putting $x_{\epsilon}\coloneqq-\epsilon y(1+y)^{-1}$ for some $y\geq0$
and $\epsilon\in]0,1[,$ we have $|x_{\epsilon}|\leq\epsilon,$$1+x_{\epsilon}=\left(1+(1-\epsilon)y\right)\left(1+y\right)^{-1}$,
which also leads to
\begin{equation}
\sqrt{1+x_{\epsilon}}=\sqrt{\left(1+(1-\epsilon)y\right)\left(1+y\right)^{-1}}\to\sqrt{\left(1+y\right)^{-1}}.\label{eq:conv}
\end{equation}
Moreover, plugging $x_{\epsilon}$ into the series \eqref{eq:incl},
we arrive at
\begin{align*}
\sqrt{1+x_{\epsilon}} & =\sum_{n=0}^{\infty}\binom{1/2}{n}x_{\epsilon}^{n}=\sum_{n=0}^{\infty}\binom{1/2}{n}\left(-\epsilon y(1+y)^{-1}\right)^{n}\\
 & =\sum_{n=0}^{\infty}\binom{1/2}{n}\left(\left(-\epsilon\right)^{n}y^{n}(1+y)^{-n}\right).
\end{align*}
By the functional calculus for selfadjoint operators, we may replace
in the latter expression $y$ by $A^{*}A$ and $AA^{*}$, respectively.
Thus, for $\epsilon\in]0,1[$, we set
\begin{align}
B_{1,\epsilon}\coloneqq & \sum_{n=0}^{\infty}\binom{1/2}{n}\left(\left(-\epsilon\right)^{n}\left(AA^{*}\right){}^{n}\left(1+\left(AA^{*}\right)\right)^{-n}\right),\nonumber \\
B_{2,\epsilon}\coloneqq & \sum_{n=0}^{\infty}\binom{1/2}{n}\left(\left(-\epsilon\right)^{n}\left(A^{*}A\right){}^{n}\left(1+\left(A^{*}A\right)\right)^{-n}\right).\label{eq:binom_op}
\end{align}
Note that $B_{1,\epsilon}$ and $B_{2,\epsilon}$ define bounded linear
operators. Moreover, by the spectral theorem (write $AA^{*}$ and
$A^{*}A$ as multiplication operators in a suitable $L^{2}$-space),
we get invoking \eqref{eq:conv}
\begin{equation}
B_{1,\epsilon}\to\sqrt{\left(1+AA^{*}\right)^{-1}}\text{ and }B_{2,\epsilon}\to\sqrt{\left(1+A^{*}A\right)^{-1}}\label{eq:conv_op}
\end{equation}
as $\epsilon\to1$ in the strong operator topology. Thus, for $\epsilon\in]0,1[$
we get with the help of \eqref{eq:powers_of_res} and \eqref{eq:binom_op}:
\begin{align*}
B_{1,\epsilon}A & =\sum_{n=0}^{\infty}\binom{1/2}{n}\left(\left(-\epsilon\right)^{n}\left(AA^{*}\right){}^{n}\left(1+\left(AA^{*}\right)\right)^{-n}\right)A\\
 & \subseteq\sum_{n=0}^{\infty}\binom{1/2}{n}\left(\left(-\epsilon\right)^{n}\left(AA^{*}\right){}^{n}A\left(1+\left(A^{*}A\right)\right)^{-n}\right)\\
 & =\sum_{n=0}^{\infty}\binom{1/2}{n}\left(\left(-\epsilon\right)^{n}A\left(A^{*}A\right)^{n}\left(1+\left(A^{*}A\right)\right)^{-n}\right)\\
 & =A\sum_{n=0}^{\infty}\binom{1/2}{n}\left(\left(-\epsilon\right)^{n}\left(A^{*}A\right)^{n}\left(1+\left(A^{*}A\right)\right)^{-n}\right)\\
 & =AB_{2,\epsilon}.
\end{align*}
Thus, the closedness of $A$, together with \eqref{eq:conv_op} yields
the asserted inclusion \eqref{eq:incl}. 
\end{proof}
Another fact used in the following is mentioned in the next proposition.
\begin{prop}
\label{prop:kappaA}Let $H_{0}$, $H_{1}$ Hilbert spaces, $A\colon D(A)\subseteq H_{0}\to H_{1}$
densely defined, closed, linear, $\kappa\in L(H_{1})$ with $0\in\rho(\kappa)$.
Then $\kappa A$ is densely defined and closed and we have 
\[
\left(\kappa A\right)^{*}=A^{*}\kappa^{*}.
\]
\end{prop}
\begin{proof}
The operator $\kappa A$ is clearly densely defined. Moreover, if
$(\phi_{n})_{n}$ is a sequence in $D(A)$ such that $(\phi_{n})_{n}$
and $(\kappa A\phi_{n})_{n}$ are convergent to $\psi\in H_{0}$ and
$\eta\in H_{1}$, we infer, by the continuous invertibility of $\kappa$
and the closedness of $A$, $\psi\in D(A)$ and $A\psi=\kappa^{-1}\psi$.
Hence, $\kappa A$ is closed. The equality $\left(\kappa A\right)^{*}=A^{*}\kappa^{*}$
is also easy.
\end{proof}
Next, we briefly recall the functional analytic setting in which we
are going to discuss the two-temperature model later on. A more detailed
discussion can be found in \cite{Picard2009,key-4} or (particularly
concerning the time-derivative) in \cite{Kalauch2014}. See also \cite{Picard2005}. 
\begin{defn}
Let $\nu>0$, $H$ Hilbert space. Define $L_{\nu}^{2}\left(\mathbb{R},H\right)$
to be the space of (equivalence classes of) square integrable functions
$f\colon\mathbb{R}\to H$ with respect to the measure with Lebesgue
density $x\mapsto e^{-2\nu x}$. Denote the space of $L_{\nu}^{2}$-functions
$f$ with distributional derivative $f'$ representable as $L_{\nu}^{2}\left(\mathbb{R},H\right)$-function
by $H_{\nu,1}\left(\mathbb{R},H\right)$. Define
\[
\partial_{0}\colon H_{\nu,1}\left(\mathbb{R},H\right)\subseteq L_{\nu}^{2}\left(\mathbb{R},H\right)\to L_{\nu}^{2}\left(\mathbb{R},H\right),f\mapsto f'.
\]

\end{defn}
Note that we will not notationally distinguish between the time-derivative
realized as an operator in $L_{\nu}^{2}(\mathbb{R},H_{1})$ and $L_{\nu}^{2}(\mathbb{R},H_{2})$
for possibly different Hilbert spaces $H_{1}$ and $H_{2}.$ The reason
of introducing this particularly weighted $L^{2}$-space is the fact
that $\partial_{0}$ becomes a \emph{continuously invertible} operator.
In fact, one has $\|\partial_{0}^{-1}\|\leq1/\nu$, see \cite{Kalauch2014}.

For a closed and densely defined linear operator $C\colon D(C)\subseteq H_{0}\to H_{1}$
between the Hilbert spaces $H_{0}$ and $H_{1}$, the lifted operator
as an abstract multiplication operator from $L_{\nu}^{2}\left(\mathbb{R},H_{0}\right)$
to $L_{\nu}^{2}\left(\mathbb{R},H_{1}\right)$ will be denoted by
the same notation. With these conventions, we can come to (a special
case of) the solution theory first established in \cite{Picard2009}.
We mention here possible generalizations to non-autonomous (\cite{PTWW2013,Waurick2015})
or non-linear frameworks (\cite{Tr2012,TW2014}). Denoting the range
of an operator $M_{0}$ by $R\left(M_{0}\right)$ and its kernel by
$N\left(M_{0}\right)$ we recall the following general solution theory
result from \cite{Picard2009,key-4}.
\begin{thm}
\label{thm:Solthy}Let $H$ Hilbert space, $M_{0}=M_{0}^{*},M_{1}\in L(H)$,
$A\colon D(A)\subseteq H\to H$ skew-selfadjoint. Assume there exists
$c>0$ such that $\left\langle M_{0}\phi,\phi\right\rangle \geq c\langle\phi,\phi\rangle$
and $\Re\left\langle M_{1}\psi,\psi\right\rangle \geq c\langle\psi,\psi\rangle$
for all $\phi\in\overline{R(M_{0})}$, $\psi\in N(M_{0})$. Then there
exists $\nu_{0}\geq0$ such that for all $\nu>\nu_{0}$ the operator
sum
\[
\mathcal{B}\coloneqq\partial_{0}M_{0}+M_{1}+A
\]
is closable as an operator in $L_{\nu}^{2}\left(\mathbb{R},H\right)$
and the closure $\overline{\mathcal{B}}$ is continuously invertible
in $L_{\nu}^{2}\left(\mathbb{R},H\right)$. Moreover, $\overline{\mathcal{B}}^{-1}$
is \emph{causal} in the sense that given $f\in L_{\nu}^{2}\left(\mathbb{R},H\right)$
with the property that $f=0$ on $(-\infty,a]$ for some $a\in\mathbb{R}$,
then $\overline{\mathcal{B}}^{-1}f=0$ on $(-\infty,a]$.
\end{thm}
The latter theorem tells us that the non-homogeneous problem $\overline{\mathcal{B}}u=f$
admits a solution for all $f\in L_{\nu}^{2}\left(\mathbb{R},H\right)$
given $\nu$ sufficiently large. In \cite{key-4} it has been shown
how to invoke initial value problems in this context. Note that it
is also possible to show that the solution $u$ does \emph{not }depend
on the parameter $\nu$, that is, let $\mu,\nu>0$ be sufficiently
large then the solution operators $\overline{\mathcal{B}_{\nu}}^{-1}$
and $\overline{\mathcal{B}_{\mu}}^{-1}$ established in $L_{\nu}^{2}\left(\mathbb{R},H\right)$
and $L_{\mu}^{2}\left(\mathbb{R},H\right)$, respectively, coincide
on the intersection of the respective domain, that is, on $L_{\nu}^{2}\left(\mathbb{R},H\right)\cap L_{\mu}^{2}\left(\mathbb{R},H\right)$. 

Later on, we will also need the following operations $\mathrm{skew}:\mathbb{C}^{3\times3}\to\mathbb{C}^{3\times3},A\mapsto\frac{1}{2}\left(A-A^{T}\right)$
and $\mathrm{sym}:\mathbb{C}^{3\times3}\to\mathbb{C}^{3\times3},A\mapsto\frac{1}{2}\left(A+A^{T}\right)$.

\section{The two-temperature model\label{sec:2tt}}

In this section, we shall have a deeper look into the two-temperature
model found in \cite{Chen-Gurtin-1968}. For this, however, we have
to introduce several vector analytical operators. In the whole section,
we assume we are given an open set $\Omega\subseteq\mathbb{R}^{n}$.
\begin{defn}
We denote by $\overset{\circ}{C}_{\infty}\left(\Omega\right)$ the
set of smooth functions with compact support. Then, we define, as
usual, $\Grad\Phi$ to be the symmetric part of the\uline{ }$3\times3$-matrix-valued
derivative of a smooth vector field $\Phi$, $\grad\phi$ be the gradient
of a smooth function $\phi$ and $\Div\Psi$ and $\dive\psi$ be the
row-wise and the usual divergence for a smooth matrix-valued functions
$\Psi$ and a smooth vector-valued function $\psi,$ respectively.
Reusing the notation $\Grad,$ $\grad,\Div$ and $\dive$ for the
respective $L^{2}(\Omega)$-realizations, we further define $\overset{\circ}{\Grad}\;:=\overline{\Grad\Big|_{\overset{\circ}{C}_{\infty}\left(\Omega\right)^{3}}}$,
$\overset{\circ}{\Div}\;:=\overline{\Div|_{\mathrm{sym}\left[\overset{\circ}{C}_{\infty}\left(\Omega\right)^{3\times3}\right]}}$,
$\overset{\circ}{\grad}\;:=\overline{\grad\Big|_{\overset{\circ}{C}_{\infty}\left(\Omega\right)}}$,
$\overset{\circ}{\dive}\;:=\overline{\dive\Big|_{\overset{\circ}{C}_{\infty}\left(\Omega\right)^{3}}}$
and their respective $L^{2}\left(\Omega\right)$-type adjoints $-\Div\;:=\left(\Grad\Big|_{\overset{\circ}{C}_{\infty}\left(\Omega\right)^{3}}\right)^{*}$,
$-\Grad\;:=\left(\Div\Big|_{\mathrm{sym}\left[\overset{\circ}{C}_{\infty}\left(\Omega\right)^{3\times3}\right]}\right)^{*}$
, $-\dive\;:=\left(\grad\Big|_{\overset{\circ}{C}_{\infty}\left(\Omega\right)}\right)^{*}$,
$-\grad\;:=\left(\dive\Big|_{\overset{\circ}{C}_{\infty}\left(\Omega\right)^{3}}\right)^{*}$.
Note that here $\Div$ maps from and $\Grad$ maps into the Hilbert
space $L_{\mathrm{sym}}^{2}\left(\Omega\right)\coloneqq L^{2}\left(\Omega,\mathrm{sym}\left[\mathbb{C}^{3\times3}\right]\right)$
of $3\times3$-symmetric-matrix valued $L^{2}$-type mappings.
\end{defn}
In the so-called two-temperature models of Chen and Gurtin \cite{Chen-Gurtin-1968},
apart from the temperature $\theta$ another temperature $\phi$,
the conductive temperature, is introduced (together with a reference
temperature $T_{\text{0 }}\in\oi0\infty$) such that
\begin{equation}
\theta-\left(\phi-T_{0}\right)=-\alpha\dive q.\label{eq:2}
\end{equation}
Here $\alpha\in\oi0\infty$ is a parameter, called the \emph{two-temperature
parameter}. Assuming homogeneous Dirichlet boundary conditions, Fourier's
law is then formulated in terms of the conductive temperature as
\begin{align}
q & =-\kappa\interior\grad\left(\phi-T_{0}\right),\label{eq:Fou}
\end{align}
where $\kappa\in L(L^{2}(\Omega)^{3})$ is a selfadjoint operator
with $\kappa\geq c>0$. In addition, the two-temperature system consists
of the heat equation with mass density $\rho_{0}\in L^{\infty}(\Omega),\,\rho_{0}\geq c_{0}>0$,
that is, 
\begin{align*}
\partial_{0}\left(\rho_{0}T_{0}\eta\right)+\dive q & =\rho_{0}Q
\end{align*}
or -- for our purposes -- more conveniently
\begin{align}
\partial_{0}\left(\rho_{0}\eta\right)+\dive\left(q/T_{0}\right) & =\rho_{0}Q/T_{0},\label{eq:heat_eq}
\end{align}
where $q$ is the heat flux as in \eqref{eq:Fou}, $\eta$ is the
entropy and $Q$ is the heat source. For the entropy $\eta$ we have
the material law relating the entropy to the temperature $\theta$
and the strain tensor $\mathcal{E}=\interior\Grad u,$ $u$ the displacement,
\begin{align}
\rho_{0}T_{0}\eta & =\rho_{0}\lambda\theta+T_{0}\gamma^{*}\mathcal{E}\label{eq:entropy}
\end{align}
for some scalar $\lambda>0,$ an operator $\gamma\in L(L^{2}(\Omega),L_{\mathrm{sym}}^{2}(\Omega))$.
Next, the strain tensor $\mathcal{E}=\interior\Grad u$ is related
to the stress tensor $\sigma$ and the temperature via the elasticity
tensor $C=C^{*}\in L\left(L_{\mathrm{sym}}^{2}(\Omega)\right)$ being
strictly positive definite and $\gamma$ in the following way
\begin{equation}
\mathcal{E}=C^{-1}\sigma+C^{-1}\gamma\theta.\label{eq:stress-strain}
\end{equation}
The two-temperature model is completed by the balance of momentum
\begin{equation}
\rho_{0}\partial_{0}^{2}u-\Div\sigma=\rho_{0}F\label{eq:balance_of_mom}
\end{equation}
for some given external force $F$.

In the following, we will show that Theorem \ref{thm:Solthy} is applicable
to the equations \eqref{eq:2}, \eqref{eq:Fou}, \eqref{eq:heat_eq},
\eqref{eq:entropy}, \eqref{eq:stress-strain}, \eqref{eq:balance_of_mom}.
Hence, the Hilbert space setting introduced in the previous section
provides a functional analytic framework such that for all right hand
sides $F$ and $Q$ there exists a unique solution to the two-temperature
model depending continuously on $F$ and $Q$. So, the task to be
solved in the next lines is to find the right unknowns and, hence,
the right operators $M_{0}$, $M_{1}$ and $A$ making Theorem \ref{thm:Solthy}
applicable.

It should be noted that our reformulation of the two-temperature model
reveals that the introduction of the second temperature transforms
the heat equation into an \emph{ordinary} differential equation with
infinite dimensional state space.

A first step towards our main goal in this section is the following
observation yielded by \eqref{eq:2} and \eqref{eq:Fou}:
\begin{prop}
\label{prop:q_theta}Let $\kappa=\kappa^{*}\in L(L^{2}(\Omega)^{3})$
be strictly positive definite. Assume that $T_{\text{0 }},\alpha\in\oi0\infty$
and $\theta\in L^{2}(\Omega)$ and $q\in D(\dive),$ $\phi\in D(\interior\grad)$
satisfy \eqref{eq:2} and \eqref{eq:Fou}. Then with%
\footnote{Of course here $\sqrt{\alpha}\kappa\sqrt{\alpha}=\alpha\kappa$, but
we prefer to write it in this more symmetric fashion, since in the
eventual first order model equations $\alpha$ can be chosen more
generally, i.e. as a continuous, selfadjoint, strictly positive definite
operator, without affecting well-posedness. Also $\kappa$ will be
allowed to be a continuous, selfadjoint, strictly positive definite
operator.%
} $\kappa_{\alpha}\coloneqq\sqrt{\alpha}\kappa\sqrt{\alpha}$ we have
\begin{align}
\sqrt{1-\sqrt{\kappa_{\alpha}}\interior\grad\dive\sqrt{\kappa_{\alpha}}}\sqrt{\kappa}^{-1}q & =-\sqrt{\kappa}\interior\grad\sqrt{1-\dive\kappa_{\alpha}\interior\grad}^{-1}\theta.\label{eq:FouFou}
\end{align}
\end{prop}
\begin{proof}
Plugging in Fourier's law we can rewrite \eqref{eq:2} as
\begin{equation}
\theta=\left(1-\dive\kappa_{\alpha}\interior\grad\right)\left(\phi-T_{0}\right).\label{eq:2-1}
\end{equation}
The operator $\sqrt{\kappa_{\alpha}}\interior\grad:D(\interior\grad)\subseteq L^{2}(\Omega)\to L^{2}(\Omega)^{3}$
is a closed densely defined linear operator, since $\kappa$ and hence,
$\sqrt{\kappa_{\alpha}}$ are boundedly invertible, see Proposition
\ref{prop:kappaA}. Moreover, its adjoint is given by $\left(\interior\grad\right)^{\ast}\sqrt{\kappa_{\alpha}}=-\dive\sqrt{\kappa_{\alpha}}$
(Proposition \ref{prop:kappaA}) and thus, $-\dive\kappa_{\alpha}\interior\grad$
is a selfadjoint, non-negative operator. In particular, $1-\dive\kappa_{\alpha}\interior\grad$
is boundedly invertible. Hence, rephrasing \eqref{eq:Fou} in terms
of the temperature $\theta$ we are led to
\begin{align*}
q & =-\kappa\interior\grad\left(1-\dive\kappa_{\alpha}\interior\grad\right)^{-1}\left(1-\dive\kappa_{\alpha}\interior\grad\right)\left(\phi-T_{0}\right)\\
 & =-\kappa\interior\grad\left(1-\dive\kappa_{\alpha}\interior\grad\right)^{-1}\theta.
\end{align*}
Next, applying Proposition \ref{prop:comm_with_sqr} to $A\coloneqq\sqrt{\kappa_{\alpha}}\interior\grad$
we infer
\[
\sqrt{1-\sqrt{\kappa_{\alpha}}\interior\grad\dive\sqrt{\kappa_{\alpha}}}^{-1}\sqrt{\kappa_{\alpha}}\interior\grad\subseteq\sqrt{\kappa_{\alpha}}\interior\grad\sqrt{1-\dive\kappa_{\alpha}\interior\grad}^{-1}
\]
and
\[
\sqrt{1-\dive\kappa_{\alpha}\interior\grad}^{-1}\dive\sqrt{\kappa_{\alpha}}\subseteq\dive\sqrt{\kappa_{\alpha}}\sqrt{1-\sqrt{\kappa_{\alpha}}\interior\grad\dive\sqrt{\kappa_{\alpha}}}^{-1},
\]
which leads us to rewrite Fourier's law as
\begin{align*}
 & \sqrt{\sqrt{\alpha}^{-1}\kappa\sqrt{\alpha}^{-1}}^{-1}q=\\
 & =-\sqrt{\kappa_{\alpha}}\interior\grad\sqrt{1-\dive\kappa_{\alpha}\interior\grad}^{-1}\sqrt{1-\dive\kappa_{\alpha}\interior\grad}^{-1}\theta\\
 & =-\sqrt{1-\sqrt{\kappa_{\alpha}}\interior\grad\dive\sqrt{\kappa_{\alpha}}}^{-1}\sqrt{\kappa_{\alpha}}\interior\grad\sqrt{1-\dive\kappa_{\text{\ensuremath{\alpha}}}\interior\grad}^{-1}\theta,
\end{align*}
yielding the assertion.
\end{proof}
With the latter observation, we are in the position of rewriting the
two-temperature model as a system in the spirit of Theorem \ref{thm:Solthy}:
\begin{thm}
\label{thm:wp2tt}Let $\kappa=\kappa^{*}\in L(L^{2}(\Omega)^{3}),$
$C=C^{*}\in L(L_{\mathrm{sym}}^{2}(\Omega))$, $\gamma\in L(L_{\mathrm{sym}}^{2}(\Omega),L^{2}(\Omega))$,
$\rho_{0}\in L(L^{2}(\Omega)),$ $\lambda,\alpha,T_{0}\in\oi0\infty$.
Then the system \eqref{eq:2}, \eqref{eq:Fou}, \eqref{eq:heat_eq},
\eqref{eq:entropy}, \eqref{eq:stress-strain}, \eqref{eq:balance_of_mom}
may be rewritten into 
\begin{equation}
\left(\partial_{0}M_{0}+M_{1}+A\right)U=J\label{eq:final_system}
\end{equation}
with $\partial_{0}u=v$ and 
\[
U=\left(\begin{array}{c}
v\\
\sigma\\
\theta\\
\sqrt{1-\sqrt{\kappa_{\alpha}}\interior\grad\dive\sqrt{\kappa_{\alpha}}}\sqrt{\kappa}^{-1}q/T_{0}
\end{array}\right),\quad J=\left(\begin{array}{c}
\rho_{0}F\\
0\\
\rho_{0}Q/T_{0}\\
0
\end{array}\right),
\]
where $\sqrt{\alpha}\kappa\sqrt{\alpha}=\kappa_{\alpha}$, 
\[
M_{0}=\left(\begin{array}{cccc}
\rho_{0} & 0 & 0 & 0\\
0 & C^{-1} & C^{-1}\gamma & 0\\
0 & \gamma^{*}C^{-1} & \left(\rho_{0}T_{0}^{-1}\lambda+\gamma^{*}C^{-1}\gamma\right) & 0\\
0 & 0 & 0 & 0
\end{array}\right),\quad A=\left(\begin{array}{cccc}
0 & -\Div & 0 & 0\\
-\interior\Grad & 0 & 0 & 0\\
0 & 0 & 0 & 0\\
0 & 0 & 0 & 0
\end{array}\right),
\]
and 
\begin{align*}
M_{1} & =\left(\begin{array}{cccc}
0 & 0 & 0 & 0\\
0 & 0 & 0 & 0\\
0 & 0 & 0 & -M_{1,32}^{*}\\
0 & 0 & M_{1,32} & T_{0}
\end{array}\right)\\
M_{1,32} & =\overline{\sqrt{1-\sqrt{\kappa_{\alpha}}\interior\grad\dive\sqrt{\kappa_{\alpha}}}^{-1}\sqrt{\kappa}\interior\grad}\\
 & =\sqrt{\kappa}\interior\grad\sqrt{1-\dive\kappa_{\alpha}\interior\grad}^{-1}.
\end{align*}
In particular, there exists $\nu_{0}\geq0$ such that for all $\nu>\nu_{0}$
the equation in \eqref{eq:final_system} admits for every $J\in L_{\nu}^{2}\left(\mathbb{R},L^{2}(\Omega)^{3}\oplus L^{2}(\Omega)^{3\times3}\oplus L^{2}(\Omega)\oplus L^{2}(\Omega)^{3}\right)$
a unique solution $U\in D\left(\overline{\partial_{0}M_{0}+M_{1}+A}\right)\subseteq L_{\nu}^{2}\left(\mathbb{R},L^{2}(\Omega)^{3}\oplus L^{2}(\Omega)^{3\times3}\oplus L^{2}(\Omega)\oplus L^{2}(\Omega)^{3}\right)$.
The solution operator is continuous and causal.\end{thm}
\begin{proof}
Before computing that the equation $\left(\partial_{0}M_{0}+M_{1}+A\right)U=J$
is a reformulation of the two-temperature model, we establish the
well-posedness issue first. For this, note that $M_{0}=M_{0}^{*}$
and $A=-A^{*}.$ Next, we check that $M_{0}$ is strictly positive
definite on its range. For the purpose of symmetric Gauss elimination,
we define the transformation matrix
\[
S\coloneqq\left(\begin{array}{cccc}
1 & 0 & 0 & 0\\
0 & 1 & \gamma & 0\\
0 & 0 & 1 & 0\\
0 & 0 & 0 & 1
\end{array}\right).
\]
Hence, 
\[
S^{-1}=\left(\begin{array}{cccc}
1 & 0 & 0 & 0\\
0 & 1 & -\gamma & 0\\
0 & 0 & 1 & 0\\
0 & 0 & 0 & 1
\end{array}\right),\quad S^{*}=\left(\begin{array}{cccc}
1 & 0 & 0 & 0\\
0 & 1 & 0 & 0\\
0 & \gamma^{*} & 1 & 0\\
0 & 0 & 0 & 1
\end{array}\right),\quad\left(S^{-1}\right)^{*}=\left(\begin{array}{cccc}
1 & 0 & 0 & 0\\
0 & 1 & 0 & 0\\
0 & -\gamma^{*} & 1 & 0\\
0 & 0 & 0 & 1
\end{array}\right).
\]
We compute that
\begin{align*}
 & \left(S^{-1}\right)^{*}M_{0}S^{-1}\\
 & =\left(\begin{array}{cccc}
1 & 0 & 0 & 0\\
0 & 1 & 0 & 0\\
0 & -\gamma^{*} & 1 & 0\\
0 & 0 & 0 & 1
\end{array}\right)\left(\begin{array}{cccc}
\rho_{0} & 0 & 0 & 0\\
0 & C^{-1} & C^{-1}\gamma & 0\\
0 & \gamma^{*}C^{-1} & \left(\rho_{0}T_{0}^{-1}\lambda+\gamma^{*}C^{-1}\gamma\right) & 0\\
0 & 0 & 0 & 0
\end{array}\right)\left(\begin{array}{cccc}
1 & 0 & 0 & 0\\
0 & 1 & -\gamma & 0\\
0 & 0 & 1 & 0\\
0 & 0 & 0 & 1
\end{array}\right)\\
 & =\left(\begin{array}{cccc}
\rho_{0} & 0 & 0 & 0\\
0 & C^{-1} & 0 & 0\\
0 & 0 & \rho_{0}T_{0}^{-1}\lambda & 0\\
0 & 0 & 0 & 0
\end{array}\right).
\end{align*}
Next, as bijective transformation $S$ reduces the space 
\[
R\coloneqq L^{2}(\Omega)^{3}\oplus L^{2}(\Omega)^{3\times3}\oplus L^{2}(\Omega)\oplus\{0\},
\]
 we infer $R(M_{0})=R$. Moreover, for $\phi\in R$ we compute
\begin{align*}
\langle M_{0}\phi,\phi\rangle & =\langle M_{0}S^{-1}S\phi,S^{-1}S\phi\rangle\\
 & =\langle\left(S^{-1}\right)^{*}M_{0}S^{-1}S\phi,S\phi\rangle\\
 & \geq\tilde{c}\langle\phi,\phi\rangle
\end{align*}
for some $\tilde{c}>0$. On $N(M_{0}),$ the operator $\Re M_{1}$,
the real part of $M_{1},$ is given by multiplication by $T_{0}>0$.
Hence, the assertion concerning well-posedness follows, once we have
established that $M_{1}$ defines a bounded linear operator. This,
however, is a direct consequence of Proposition \ref{prop:comm_with_sqr}:
Indeed, 
\begin{align*}
 & \left|\sqrt{\kappa}\interior\grad\sqrt{1-\dive\kappa_{\alpha}\interior\grad}^{-1}\phi\right|_{0}\\
 & =\frac{1}{\sqrt{\alpha}}\left|\left|\sqrt{\kappa_{\alpha}}\interior\grad\right|\sqrt{1+\left|\sqrt{\kappa_{\alpha}}\interior\grad\right|^{2}}^{-1}\phi\right|_{0}\\
 & \leq\frac{1}{\sqrt{\alpha}}\left|\phi\right|_{0}\quad\left(\phi\in L^{2}(\Omega)\right).
\end{align*}

As a next step we proceed in showing that the two-temperature model
admits the asserted reformulation. For this, in turn, it suffices
to observe the following consequence of the equations \eqref{eq:entropy}
and \eqref{eq:stress-strain}:
\begin{align*}
\rho_{0}T_{0}\eta & =\rho_{0}\lambda\theta+T_{0}\gamma^{*}\mathcal{E}\\
 & =\rho_{0}\lambda\theta+T_{0}\gamma^{*}\left(C^{-1}\sigma+C^{-1}\gamma\theta\right).
\end{align*}
Hence, 
\[
\rho_{0}\eta=\left(\rho_{0}T_{0}^{-1}\lambda+\gamma^{*}C^{-1}\gamma\right)\theta+\gamma^{*}C^{-1}\sigma.
\]
Moreover, from $\mathcal{E}=\interior\Grad u$ and $\partial_{0}u=v$
it follows that
\begin{align*}
\partial_{0}\mathcal{E}-\interior\Grad v & =0
\end{align*}
 and the balance of momentum \eqref{eq:balance_of_mom} reads as
\[
\rho_{0}\partial_{0}v-\Div\sigma=\rho_{0}F.
\]
Recalling \eqref{eq:FouFou} from Proposition \ref{prop:q_theta},
we note that{\scriptsize{}
\begin{align*}
 & \dive\left(q/T_{0}\right)=\\
 & =\dive\sqrt{\kappa}\sqrt{1-\sqrt{\kappa_{\alpha}}\interior\grad\dive\sqrt{\kappa_{\alpha}}}^{-1}\left(\sqrt{1-\sqrt{\kappa_{\alpha}}\interior\grad\dive\sqrt{\kappa_{\alpha}}}\sqrt{\kappa}^{-1}q/T_{0}\right),
\end{align*}
}which eventually establishes the assertion.%

\end{proof}
Note that $M_{1,32}$ has moved from its place in $A$ for the limit
case $\alpha=0$ to the material law.
\begin{rem}
Symbolizing non-vanishing entries in the block operator matrices under
consideration by $\bigstar$, clearly, the pattern of $M_{0}$ is
\[
M_{0}=\left(\begin{array}{cccc}
\bigstar & 0 & 0 & 0\\
0 & \bigstar & \bigstar & 0\\
0 & \bigstar & \bigstar & 0\\
0 & 0 & 0 & 0
\end{array}\right).
\]
But the pattern of $M_{1}$ is
\[
\Re M_{1}=\left(\begin{array}{cccc}
0 & 0 & 0 & 0\\
0 & 0 & 0 & 0\\
0 & 0 & 0 & 0\\
0 & 0 & 0 & \bigstar
\end{array}\right),\:\Im M_{1}=\left(\begin{array}{cccc}
0 & 0 & 0 & 0\\
0 & 0 & 0 & 0\\
0 & 0 & 0 & \bigstar\\
0 & 0 & \bigstar & 0
\end{array}\right).
\]
Moreover, 
\[
A=\left(\begin{array}{cccc}
0 & -\Div & 0 & 0\\
-\interior\Grad & 0 & 0 & 0\\
0 & 0 & 0 & 0\\
0 & 0 & 0 & 0
\end{array}\right).
\]
We see that the system has partly been turned into an ode in an infinite
dimensional state space.
\end{rem}
{}

\section{A Two-Temperature, Two-Strain Model.\label{sec:2tt-2s}}

In this section, we shall elaborate briefly on the possibility of
developing an alternative model, such that the whole pde-part in the
two-temperature model discussed in the previous section vanishes.
We start with basically the same model as in Theorem \ref{thm:wp2tt}.
As a preparation for deriving the two-temperature, two-strain model,
we consider first the following system, which is unitarily congruent
to the one in Theorem \ref{thm:wp2tt}: 

\begin{align*}
 & \partial_{0}\Big(\left(\begin{array}{cccc}
\rho_{0} & 0 & 0 & 0\\
0 & 1 & C^{-1/2}\gamma & 0\\
0 & \gamma^{*}C^{-1/2} & \left(\rho_{0}T_{0}^{-1}\lambda+\gamma^{*}C^{-1}\gamma\right) & 0\\
0 & 0 & 0 & 0
\end{array}\right)\times\\
 & \times\left(\begin{array}{c}
v\\
C^{-1/2}\sigma\\
\theta\\
\left(1-\sqrt{\kappa_{\alpha}}\interior\grad\dive\sqrt{\kappa_{\alpha}}\right)^{1/2}\sqrt{\kappa}^{-1}q/T_{0}
\end{array}\right)\Big)+\\
 & \quad+\left(\begin{array}{cccc}
0 & 0 & 0 & 0\\
0 & 0 & 0 & 0\\
0 & 0 & 0 & -M_{1,32}^{*}\\
0 & 0 & M_{1,32} & T_{0}
\end{array}\right)\left(\begin{array}{c}
v\\
C^{-1/2}\sigma\\
\theta\\
\left(1-\sqrt{\kappa_{\alpha}}\interior\grad\dive\sqrt{\kappa_{\alpha}}\right)^{1/2}\sqrt{\kappa}^{-1}q/T_{0}
\end{array}\right)+\\
 & \quad+A\left(\begin{array}{c}
v\\
C^{-1/2}\sigma\\
\theta\\
\left(1-\sqrt{\kappa_{\alpha}}\interior\grad\dive\sqrt{\kappa_{\alpha}}\right)^{1/2}\sqrt{\kappa}^{-1}q/T_{0}
\end{array}\right)\\
 & =\left(\begin{array}{c}
f\\
0\\
\rho_{0}Q/T_{0}\\
0
\end{array}\right)
\end{align*}
 with
\[
A=\left(\begin{array}{cccc}
0 & -\Div C^{1/2} & 0 & 0\\
-C^{1/2}\interior\Grad & 0 & 0 & 0\\
0 & 0 & 0 & 0\\
0 & 0 & 0 & 0
\end{array}\right),
\]
where 
\[
M_{1,32}=\sqrt{\kappa}\interior\grad\sqrt{1-\dive\kappa_{\alpha}\interior\grad}^{-1}.
\]
Taking this as a starting point and substituting $C_{\beta}\coloneqq\sqrt{\beta}C\sqrt{\beta}$
for some $\beta>0$, we may propose analogously a similar modification
of the elastic part yielding{\small{} }{\small \par}

{\small{}
\begin{align*}
 & \partial_{0}\Big(\left(\begin{array}{cccc}
\rho_{0} & 0 & 0 & 0\\
0 & 1 & C^{-1/2}\gamma & 0\\
0 & \gamma^{*}C^{-1/2} & \left(\rho_{0}T_{0}^{-1}\lambda+\gamma^{*}C^{-1}\gamma\right) & 0\\
0 & 0 & 0 & 0
\end{array}\right)\times\\
 & \quad\times\left(\begin{array}{c}
v\\
\sqrt{1-\sqrt{C_{\beta}}\interior\Grad\Div\sqrt{C_{\beta}}}\:\sqrt{C}^{-1}\sigma\\
\theta\\
\sqrt{1-\sqrt{\sqrt{\alpha}\kappa\sqrt{\alpha}}\interior\grad\dive\sqrt{\sqrt{\alpha}\kappa\sqrt{\alpha}}}\sqrt{\kappa}^{-1}q/T_{0}
\end{array}\right)\Big)+\\
 & \quad+\left(\begin{array}{ccccc}
0 & -M_{1,10}^{*} & 0 & 0 & 0\\
M_{1,10} & 0 & 0 & 0 & 0\\
0 & 0 & 0 & 0 & 0\\
0 & 0 & 0 & 0 & -M_{1,32}^{*}\\
0 & 0 & 0 & M_{1,32} & T_{0}
\end{array}\right)\left(\begin{array}{c}
v\\
\sqrt{1-\sqrt{C_{\beta}}\interior\Grad\Div\sqrt{C_{\beta}}}\:\sqrt{C}^{-1}\sigma\\
\theta\\
\sqrt{1-\sqrt{\kappa_{\alpha}}\interior\grad\dive\sqrt{\kappa_{\alpha}}}\sqrt{\kappa}^{-1}q/T_{0}
\end{array}\right)\\
 & =\left(\begin{array}{c}
f\\
0\\
\rho_{0}Q/T_{0}\\
0
\end{array}\right)
\end{align*}
} where now
\[
M_{1,10}=-\sqrt{C}\interior\grad\sqrt{1-\sqrt{C_{\beta}}\interior\Grad\Div\sqrt{C_{\beta}}}^{-1}
\]
 and 
\[
M_{1,32}=\sqrt{\kappa}\interior\grad\sqrt{1-\dive\kappa_{\alpha}\interior\grad}^{-1}.
\]

Clearly, the pattern of $M_{0}$ is still 
\[
M_{0}=\left(\begin{array}{cccc}
\bigstar & 0 & 0 & 0\\
0 & \bigstar & \bigstar & 0\\
0 & \bigstar & \bigstar & 0\\
0 & 0 & 0 & 0
\end{array}\right).
\]
But the pattern of $M_{1}$ is now
\[
\mathrm{sym}M_{1}=\left(\begin{array}{cccc}
0 & 0 & 0 & 0\\
0 & 0 & 0 & 0\\
0 & 0 & 0 & 0\\
0 & 0 & 0 & \bigstar
\end{array}\right),\:\mathrm{skew}M_{1}=\left(\begin{array}{cccc}
0 & \bigstar & 0 & 0\\
\bigstar & 0 & 0 & 0\\
0 & 0 & 0 & \bigstar\\
0 & 0 & \bigstar & 0
\end{array}\right).
\]
Moreover, 
\[
A=\left(\begin{array}{cccc}
0 & 0 & 0 & 0\\
0 & 0 & 0 & 0\\
0 & 0 & 0 & 0\\
0 & 0 & 0 & 0
\end{array}\right).
\]
We see that the system has completely been turned into an abstract
ode.
\begin{rem}
~
\begin{itemize}
\item Taking the general perspective used here into account, more complex
materials, for instance, the Maxwell-Cattaneo-Vernotte (MCV) model
of heat conduction \cite{key-1,key-2,key-9}, are easily possible
to include all these models as introduced in \cite{key-21}. For implementing
the MCV model we merely have to take $M_{0}$ with the pattern
\[
M_{0}=\left(\begin{array}{cccc}
\bigstar & 0 & 0 & 0\\
0 & \bigstar & \bigstar & 0\\
0 & \bigstar & \bigstar & 0\\
0 & 0 & 0 & \bigstar
\end{array}\right)
\]
as strictly positive definite.
\item Moreover, if we change the parameter $\alpha$ (and $\beta$) to be
bounded, selfadjoint, strictly positive definite operators in an appropriate
Hilbert spaces, we gain further flexibility for material modeling
within the framework of the first order system.
\item Given the intricate rationale used in deriving the model in the first
place it is somewhat disappointing to see that it merely serves to
approximate a pde by an ode, which of course is always possible, compare
e.g. the Yosida approximation or the above strategy, which amounts
to replacing an unbounded skew-selfadjoint operator $A$ by the bounded
skew-selfadjoint operator $A\sqrt{1-\alpha A^{2}}^{-1}=\overline{\sqrt{1-\alpha A^{2}}^{-1}A}$,
$\alpha\in\oi0\infty$. 
\end{itemize}

\section{An alternative two-temperature model\label{sec:a-2tt}}

In this section, we will make an attempt to establish an alternative
two-temperature model from a purely structural point of view. For
this, we proceed as follows:

Note that a transition to the ode setting can also be achieved for
example by approximating $A$ with $A\left(1+\epsilon A\right)^{-1}$,
(Yosida approximation). Indeed,
\begin{equation}
A\left(1+\epsilon A\right)^{-1}\overset{\epsilon\to0}{\to}A\label{eq:strong_conv}
\end{equation}
point-wise on $D\left(A\right)$.

This way the occurrence of a square root (of inverses) of unbounded
operators can be avoided. We assume the conditions of Theorem \ref{thm:wp2tt}.
For notational convenience we set $D\coloneqq\sqrt{\kappa}\interior\grad$.
Applying the idea of using \eqref{eq:strong_conv} to our initial
two-temperature model yields
\begin{align*}
 & \partial_{0}\Big(\left(\begin{array}{cccc}
\rho_{0} & 0 & 0 & 0\\
0 & C^{-1} & C^{-1}\gamma & 0\\
0 & \gamma^{*}C^{-1} & \left(\rho_{0}T_{0}^{-1}\lambda+\gamma^{*}C^{-1}\gamma\right) & 0\\
0 & 0 & 0 & 0
\end{array}\right)\times\\
 & \quad\times\left(\begin{array}{c}
v\\
\sigma\\
\theta\\
\left(1+\epsilon^{2}DD^{*}\right)\sqrt{\kappa}^{-1}q/T_{0}+\epsilon D\theta
\end{array}\right)\Big)+\\
 & \quad+\Big(\left(\begin{array}{cccc}
0 & 0 & 0 & 0\\
0 & 0 & 0 & 0\\
0 & 0 & \epsilon D^{*}D\left(1+\epsilon^{2}D^{*}D\right)^{-1} & -D^{*}\left(1+\epsilon^{2}DD^{\ast}\right)^{-1}\\
0 & 0 & D\left(1+\epsilon^{2}D^{*}D\right)^{-1} & \epsilon DD^{*}\left(1+\epsilon^{2}DD^{*}\right)^{-1}+T_{0}
\end{array}\right)\times\\
 & \quad\times\left(\begin{array}{c}
v\\
\sigma\\
\theta\\
\left(1+\epsilon^{2}DD^{*}\right)\sqrt{\kappa}^{-1}q/T_{0}+\epsilon D\theta
\end{array}\right)\Big)+\\
 & \quad+A\left(\begin{array}{c}
v\\
\sigma\\
\theta\\
\left(1+\epsilon^{2}DD^{\ast}\right)\sqrt{\kappa}^{-1}q/T_{0}+\epsilon D\theta
\end{array}\right)U\\
 & =\left(\begin{array}{c}
\rho_{0}F\\
0\\
\rho_{0}Q/T_{0}\\
0
\end{array}\right)
\end{align*}
with 
\[
A=\left(\begin{array}{cccc}
0 & -\Div & 0 & 0\\
-\interior\Grad & 0 & 0 & 0\\
0 & 0 & 0 & 0\\
0 & 0 & 0 & 0
\end{array}\right).
\]
This reduces indeed to
\begin{align*}
\partial_{0}\rho_{0}v-\Div\sigma & =f\\
\partial_{0}\left(\sigma+\gamma\theta\right)-C\interior\Grad v & =0\\
\partial_{0}\left(\gamma^{*}C^{-1}\sigma+\left(\rho_{0}T_{0}^{-1}\lambda+\gamma^{*}C^{-1}\gamma\right)\theta\right)-D^{*}\sqrt{\kappa}^{-1}q/T_{0} & =\rho_{0}Q/T_{0}
\end{align*}

and finally
\begin{align*}
D(1+\epsilon^{2}D^{\ast}D)^{-1}\theta+\\
+\epsilon DD^{\ast}\left(1+\epsilon^{2}DD^{\ast}\right)^{-1}\left((1+\epsilon^{2}DD^{\ast})\sqrt{\kappa}^{-1}q/T_{0}+\epsilon D\theta\right)+\\
+T_{0}\left((1+\epsilon^{2}DD^{\ast})\sqrt{\kappa}^{-1}q/T_{0}+\epsilon D\theta\right) & =0,\\
\left(\epsilon DD^{\ast}+T_{0}+T_{0}\epsilon^{2}DD^{\ast}\right)\sqrt{\kappa}^{-1}q/T_{0}+\\
+\left((1+\epsilon^{2}DD^{\ast})^{-1}+\epsilon^{2}DD^{\ast}\left(1+\epsilon^{2}DD^{\ast}\right)^{-1}+\epsilon T_{0}\right)D\theta & =0,\\
\left(1+\epsilon^{2}DD^{\ast}\right)\sqrt{\kappa}^{-1}q+\epsilon DD^{\ast}\sqrt{\kappa}^{-1}q/T_{0}+\left(1+\epsilon T_{0}\right)D\theta & =0,\\
\left(1+\epsilon^{2}DD^{\ast}\right)\sqrt{\kappa}^{-1}q+D\left(\epsilon D^{\ast}\sqrt{\kappa}^{-1}q/T_{0}+\theta+\epsilon T_{0}\theta\right) & =0.
\end{align*}
Thus,
\begin{align}
\sqrt{\kappa}^{-1}q+\left(1+\epsilon^{2}DD^{\ast}\right)^{-1}D\left(\epsilon D^{\ast}\sqrt{\kappa}^{-1}q/T_{0}+\theta+\epsilon T_{0}\theta\right) & =0,\nonumber \\
\sqrt{\kappa}^{-1}q+D\left(\left(1+\epsilon^{2}D^{\ast}D\right)^{-1}\epsilon D^{\ast}\sqrt{\kappa}^{-1}q/T_{0}+(1+\epsilon T_{0})\left(1+\epsilon^{2}D^{\ast}D\right)^{-1}\theta\right) & =0,\label{eq:Fourier}
\end{align}
and hence, defining
\[
\phi\coloneqq\left(1+\epsilon T_{0}\right)\left(1+\epsilon^{2}D^{*}D\right)^{-1}\theta+\epsilon\left(1+\epsilon^{2}D^{\ast}D\right)^{-1}D^{*}\sqrt{\kappa}^{-1}q/T_{0}+T_{0}
\]
and recalling that $D=\sqrt{\kappa}\interior\grad$, we end up with
\begin{align*}
\theta & =\left(1+\epsilon T_{0}\right)^{-1}\left(1+\epsilon^{2}D^{*}D\right)\left(\phi-T_{0}\right)-\frac{\epsilon}{T_{0}}\left(1+\epsilon T_{0}\right)^{-1}D^{*}\sqrt{\kappa}^{-1}q\\
 & =\left(1+\epsilon T_{0}\right)^{-1}\left(1-\epsilon^{2}\dive\kappa\interior\grad\right)\left(\phi-T_{0}\right)+\frac{\epsilon}{T_{0}}\left(1+\epsilon T_{0}\right)^{-1}\dive q
\end{align*}
and \eqref{eq:Fourier} gives the Fourier law
\[
q+\kappa\interior\grad\phi=0.
\]
Thus, using $-\epsilon^{2}\dive\kappa\interior\grad\theta=\epsilon^{2}\dive q$,
we get that
\begin{align*}
\theta & =(1+\epsilon T_{0})^{-1}\left(\phi-T_{0}\right)+(1+\epsilon T_{0})^{-1}\epsilon^{2}\dive q+(1+\epsilon T_{0})^{-1}\frac{\epsilon}{T_{0}}\dive q\\
 & =(1+\epsilon T_{0})^{-1}\left(\phi-T_{0}\right)+\frac{\epsilon}{T_{0}}\left((1+\epsilon T_{0})^{-1}\epsilon T_{0}\dive q+(1+\epsilon T_{0})^{-1}\dive q\right)\\
 & =\phi-T_{0}+\frac{\epsilon}{T_{0}}\left(\dive q-\frac{T_{0}^{2}}{1+\epsilon T_{0}}\phi\right)
\end{align*}

This can also be written as
\begin{equation}
\theta-\left(1-\frac{\epsilon T_{0}}{1+\epsilon T_{0}}\right)\left(\phi-T_{0}\right)=\frac{\epsilon}{T_{0}}\dive q\label{eq:final}
\end{equation}

Equation \eqref{eq:final} represents the final relation satisfied
by the two temperatures. The parameter $\epsilon$ would be an \emph{alternative
two-temperature parameter}.

\end{rem}

\end{document}